\documentclass[11pt]{amsart}
\usepackage{amssymb, a4wide, mathdots, url, graphicx}
\usepackage[usenames]{color}
\usepackage{enumerate}
\usepackage{mathtools}
\usepackage{xcolor}
\usepackage[hyperfootnotes=false, colorlinks, linkcolor={blue}, citecolor={magenta}, filecolor={blue}, urlcolor={blue}, plainpages=false, pdfpagelabels]{hyperref}
\usepackage{amsmath,amscd}
\usepackage{tikz-cd}
\usepackage{amsfonts}

\newcommand{\Ad}{\operatorname{Ad}}

\newcommand{\Lie}{\operatorname{Lie}}

\newcommand{\C}{\mathbb{C}}

\newcommand{\Hom}{\operatorname{Hom}}

\newcommand{\Span}{\operatorname{Span}}

\newtheorem{theorem}{Theorem}[section]
\newtheorem{lemma}[theorem]{Lemma}

\newtheorem{proposition}[theorem]{Proposition}%[subsection]
\theoremstyle{remark}
\newtheorem{remark}[theorem]{Remark}%[subsection]
\makeatletter
\@namedef{subjclassname@2020}{%
  \textup{2020} Mathematics Subject Classification}
\makeatother

\newcommand{\pr}{\operatorname{pr}}

\DeclareMathOperator{\diag}{diag}

\DeclareMathOperator{\Ind}{Ind}
\DeclareMathOperator{\GL}{GL}

\DeclareMathOperator{\SO}{SO}

\DeclareMathOperator{\GSpin}{GSpin}
\DeclareMathOperator{\GPin}{GPin}

\DeclareMathOperator{\Orth}{O}

\begin{document}

\title{Uniqueness of Bessel models for GSpin groups}

%\thanks{This work was supported by}

\author{Pan Yan}
\address{Department of Mathematics, The University of Arizona, Tucson, AZ 85721, USA}
\email{panyan@arizona.edu}

\date{\today}

\subjclass[2020]{Primary 11F70, 22E50; Secondary 22E45, 20G05}
\keywords{Bessel models, multiplicity one theorem, local Gan-Gross-Prasad conjecture}

\begin{abstract}
We prove the uniqueness of general Bessel models for $\mathrm{GSpin}$ groups over a local field of characteristic zero. The proof is to reduce it to the spherical case, which has been proved by Emory and Takeda \cite{EmoryTakeda2023} in the non-archimedean case and by Emory, Kim, and Maiti \cite{EmoryKimMaiti} in the archimedean case.
\end{abstract}

\maketitle

\goodbreak 

\tableofcontents

\goodbreak

\section{Introduction}

Let $V$ be a non-degenerate quadratic space over a local field $F$ of characteristic 0. 
Let $V_0\subset V$ be a non-degenerate subspace such that the orthogonal complement $V_0^\perp$ is split and has dimension $2r+1$ for a non-negative integer $r$. 
Denote by $\GSpin(V)$ and $\GSpin(V_0)$ the corresponding $\GSpin$ groups.
We have an embedding $\GSpin(V_0)\subseteq \GSpin(V)$, and their centers share the same connected component. Associated to the pair $(V_0, V)$, we define the $r$-th Bessel subgroup of $G=\GSpin(V)\times \GSpin(V_0)$ by
\begin{equation*}
	H=N_{Q_r} \rtimes \GSpin(V_0),
\end{equation*}
where $N_{Q_r}$ is the unipotent radical of the parabolic subgroup of $\GSpin(V)$ stabilizing a complete flag of isotropic subspaces determined by $V_0^{\perp}$. 
Let $\xi$ be a generic character of $H$, as defined in Section~\ref{subsection-Bessel-subgroup}. The goal of this paper is to establish the following result on the uniqueness of local Bessel models for $G$.
\begin{theorem}
\label{thm-main}
Let $V, V_0, H, G, \xi$ be as above. 
For any irreducible admissible representation $\pi$ of $\GSpin(V)$ and $\pi_0$ of $\GSpin(V_0)$, which are of Casselman-Wallach type if $F$ is archimedean, the following inequality holds:
\begin{equation*}
\dim_{\C} \Hom_{H}(\pi \otimes \pi_0, \xi)\le 1 .
\end{equation*}
\end{theorem}

When $F$ is archimedean, the symbol $\otimes$ stands for the completed projective tensor product, and the symbol $\Hom_{H}$ stands for the space of continuous $H$-intertwining maps.

Denote by $\omega_{\pi}$ and $\omega_{\pi_0}$ the central characters of $\pi$ and $\pi_0$ respectively. Note that $F^\times$ is contained in both the centers of $\GSpin(V)$ and $\GSpin(V_0)$; see Lemma~\ref{lemma-center}. If $\omega_{\pi}|_{F^\times} \cdot \omega_{\pi_0}|_{F^\times}$ is not the trivial character of $F^\times$, then $\Hom_{H}(\pi \otimes \pi_0, \xi)=0$.    See Remark~\ref{remark-central-character-compatibility}.
In the rest of the paper, we always assume that  $\omega_{\pi}|_{F^\times} \cdot \omega_{\pi_0}|_{F^\times}$ is the trivial character of $F^\times$.

We discuss two extreme cases.
When $r=0$, the subgroup $H$ is given by $H=\GSpin(V_0)$, which is diagonally embedded into $G=\GSpin(V)\times\GSpin(V_0)$, with $\xi$ being trivial. In this case, the inequality in Theorem~\ref{thm-main} recovers the uniqueness for spherical models, proved in \cite{EmoryTakeda2023} in the non-archimedean case and in \cite{EmoryKimMaiti} in the archimedean case. 
On the other hand, when $V_0=0$, we have $\GSpin(V_0)=F^\times$, which is precisely the center of $\GSpin(V)$ where $V$ is odd-dimensional; see Lemma~\ref{lemma-center}. In this case, the inequality in Theorem~\ref{thm-main} recovers the uniqueness for Whittaker models; see \cite{CasselmanHechtMilivciv2000} when $F$ is archimedean and \cite{Shalika1974} when $F$ is non-archimedean. Hence the family of Bessel models generalizes both the Whittaker model ($V_0=0$) and the spherical model ($r=0$). 

\begin{remark}
\label{Intro-remark-dim>3}
We discuss some low-rank cases.
When $\dim(V)\le 2$, $\GSpin(V)$ is abelian, hence $\pi$ and $\pi_0$ are both one-dimensional, and hence the inequality in Theorem~\ref{thm-main} is clear. If $\dim(V)=3$, then either $\dim(V_0)=2$ or $\dim(V_0)=0$. The case $(\dim(V),\dim(V_0))=(3,2)$ is known by the uniqueness of spherical models in \cite{EmoryTakeda2023} and \cite{EmoryKimMaiti}. The case $(\dim(V), \dim(V_0))=(3,0)$ is known by the uniqueness of Whittaker models in \cite{CasselmanHechtMilivciv2000} and \cite{Shalika1974}. Therefore, Theorem~\ref{thm-main} is known when $\dim(V)\le 3$. In the rest of the paper, we will always assume that $\dim(V)>3$.
\end{remark}

For classical groups, the uniqueness of Bessel models is proved in \cite{AizenbudGourevitchRallisSchiffmann2010, Waldspurger2012MultiplicityOneSO, GanGrossPrasad2012} when $F$ is non-archimedean, and in \cite{SunZhu2012, JiangSunZhu2010} when $F$ is archimedean.

We remark that Theorem~\ref{thm-main} establishes an analogue of the result in \cite[\S 15]{GanGrossPrasad2012} and completes the first step towards the local Gan-Gross-Prasad conjecture for $\GSpin$ groups.

Another important application of Theorem~\ref{thm-main} is the Euler product factorization of certain global Rankin-Selberg integrals for $\GSpin$ groups, studied in the author's work in \cite{Yan2025GGP}. As a further application of the Rankin-Selberg integrals, certain cases of the global Gan-Gross-Prasad conjecture for $\GSpin$ groups are proved in \cite{Yan2025GGP}. More recently, Bessel models are used in the construction of twisted automorphic descent from the general linear groups to odd $\GSpin$ groups in \cite{Yan2026TwistedDescent}.

We now discuss the organization of the paper as well as the idea of the proof. In Section~\ref{section-preliminaries}, we review the structure of the $\GSpin$ group using Clifford algebra and its parabolic subgroups. In Section~\ref{section-preparation-of-proof}, as a preparation of the proof of Theorem~\ref{thm-main}, we discuss the Bessel subgroup $H$ and the character $\xi$, as well as representations of $\GSpin$ groups. In Section~\ref{subsection-local-integral} we introduce a local zeta integral $\mathcal{Z}_\mu$ associated with a Bessel functional $\mu\in \Hom_H(\pi\otimes\pi_0,\xi)$, which depends on a complex parameter $s$, following an approach in \cite{JiangSunZhu2010}. Proposition~\ref{prop-integral-converge-1} asserts that if $\mu$ is nonzero, then the local zeta integral $\mathcal{Z}_\mu$ can be made nonzero. Proposition~\ref{prop-integral-converge-2} states that $\mathcal{Z}_\mu$ converges absolutely when $\mathrm{Re}(s)$ is large enough. As a consequence $\mathcal{Z}_\mu$ defines an element in the spherical model $\Hom_{\GSpin(V)}(\pi_s^\prime\otimes\pi,\mathbb{C})$, where $\pi_s^\prime$ is defined in Section~\ref{subsection-induced-rep}. In Section~\ref{subsection-proof-prop1}, we prove Proposition~\ref{prop-integral-converge-1}. Finally, in Section~\ref{subsection-proof-of-thm}, we prove Theorem~\ref{thm-main}.

\subsection{Notation}

Throughout the paper, we assume that $F$ is a local field of characteristic zero.
All representations in this paper are complex and smooth. For a representation $\pi$, we also use $\pi$ to denote its underlying space. Over archimedean fields, by an admissible representation we mean an admissible Casselman-Wallach representation. Morphisms are continuous and induction is smooth, and $\otimes$ is the completed projective tensor product, over archimedean fields.

\subsection*{Acknowledgements}
The author would like to thank Hang Xue for many helpful discussions. The author would also like to thank Dihua Jiang for valuable discussions. 
The author thanks the referee for a careful reading and valuable suggestions on the manuscript.
The author is partially supported by an AMS-Simons Travel Grant.

\section{The $\GSpin$ group and its parabolic subgroups}
\label{section-preliminaries}

\subsection{The $\GSpin$ group}
In this section we recall how the $\GSpin$ group is realized via the Clifford algebra. Our main reference is \cite{Shimura2004}, which introduces the Clifford algebras associated to quadratic spaces over arbitrary field of characteristic different from 2. 

Let $F$ be a local field of characteristic zero. 
Let $V$ be a finite-dimensional quadratic space over $F$ equipped with a non-degenerate quadratic form $q_V$, giving rise to the groups $\Orth(V)$ and $\SO(V)$:
\begin{equation*}
\begin{split}
\Orth(V)&=\{g\in \GL_F(V): q_V(gv)=q_V(v) \text{ for every $v\in V$}\},\\
\SO(V)&=\{g\in \Orth(V):\det(g)=1\}.
\end{split}
\end{equation*}
The Clifford algebra $C(V)$ associated to $(V, q_V)$ is the quotient of the tensor algebra $T(V)=\bigoplus_{d=0}^{\infty} V^{\otimes d}$
by the two sided ideal generated by all $v\otimes v - q_V(v)$, for $v\in V$. To simplify notation we also write $v_1v_2\cdots v_k$ for $v_1\otimes v_2\otimes \cdots \otimes v_k$.

The inclusion map $V\to T(V)$ induces a canonical injection $\iota: V\to C(V)$.
Identifying $v$ with $\iota(v)$ for every $v\in V$, we view $V$ as a subspace of $C(V)$, so that $v^2=q_V(v)$.
For $v_1, v_2\in V$, we set $\langle v_1, v_2\rangle: =  q_V(v_1+v_2)-q_V(v_1)-q_V(v_2)$ the induced symmetric bilinear form, so that $\langle v, v \rangle=2q_V(v)$. Similarly, for $x, y\in C(V)$, we denote $\langle x, y\rangle=(x+y)^2-x^2-y^2$. In particular, $\langle x, y\rangle$ measures if $x$ and $y$ anti-commute in $C(V)$ since
\begin{equation}
\label{eq-CliffordAlgebra-anticommute}
\langle x, y\rangle=  (x+y)^2-x^2-y^2 = xy+yx \in C(V).
\end{equation}

Denote by \( C^\pm(V) \) the even and odd part of \( C(V) \), so that $C(V)=C^+(V)\oplus C^-(V)$. 
The Clifford algebra $C(V)$ has a canonical involution $*:C(V)\to C(V)$, induced by reversing the order of products in the tensor algebra, i.e., $(v_1v_2\cdots v_k)^*=v_k\cdots v_2v_1$ for $v_1, \cdots, v_k\in V$. The canonical involution preserves both $C^+(V)$ and $C^-(V)$.
We also define
\begin{equation*}
\alpha:C(V)\to C(V), \quad \alpha(v_{+}+v_{-})=v_{+}-v_{-}, 	
\end{equation*}
where $v_{+}\in C^{+}(V)$, and $v_{-}\in C^{-}(V)$. So $\alpha$ acts as the identity on $C^{+}(V)$ and as multiplication by $-1$ on $C^{-}(V)$, and $\alpha(v_1v_2\cdots v_k)=(-1)^k v_1 v_2\cdots v_k$ for $v_1, \cdots, v_k\in V$. 
Moreover, for all $x\in C(V)$ we define the Clifford involution
\begin{equation*}
\overline{x}=(\alpha(x))^*=\alpha(x^*)
\end{equation*}
and the Clifford norm 
\begin{equation*}
N: C(V)\to C(V), \quad x\mapsto x \overline{x}.	
\end{equation*}
It is clear that $N(\lambda x)=\lambda^2 N(x)$ for all $\lambda\in F$.

The $\GPin$ group and $\GSpin$ group of $(V, q_V)$ are defined by
\begin{equation*}
\begin{split}
\GPin(V) &=\{g\in C(V)^\times: \alpha(g) V g^{-1}=V\},	\\
\GSpin(V) &=\GPin(V)\cap C^{+}(V).
\end{split}
\end{equation*}
These are also called the Clifford group and the even Clifford group in the literature. 
Since $\alpha$ acts as identity on $C^+(V)$, we have the inclusion $\GSpin(V)\subset \GPin(V)$. In fact, we have $[\GPin(V):\GSpin(V)]=2$, and $\GSpin(V)$ is the identity component of $\GPin(V)$.

There is a natural projection map 
\begin{equation}
\label{eq-pr}
\pr:\GPin(V)\to \Orth(V)	
\end{equation} 
sending $g\in \GPin(V)$ to the map $v\mapsto \alpha(g)vg^{-1}$.  The kernel of this homomorphism consists of scalars, thus we get an exact sequence
\begin{equation*}
1\to F^\times \to \GPin(V)\to \Orth(V)\to 1.	
\end{equation*}
Moreover, $\pr^{-1}(\SO(V))=\GSpin(V)$, and we have a commutative diagram
\[
\begin{tikzcd}
1 \ar[r] & F^\times \ar[r] \ar[d, equals] & \mathrm{GSpin}(V) \ar[r] \ar[d, "\subseteq"] & \mathrm{SO}(V) \ar[r] \ar[d, "\subseteq"] & 1 \\
1 \ar[r] & F^\times \ar[r] & \mathrm{GPin}(V) \ar[r] & \mathrm{O}(V) \ar[r] & 1
\end{tikzcd}
\]
where the rows are exact. It is known (see, for example, \cite[Theorem 3.1]{Meinrenken2013}) that the elements of $\GPin(V)$ are all products
\begin{equation}
\label{eq-GPin-homogeneous}
g=v_1\cdots v_k 
\end{equation}
where $v_1, \cdots, v_k\in V$ are non-isotropic. Given $g=v_1\cdots v_k\in \GPin(V)$, the corresponding element $\pr(g)\in \Orth(V)$ is a product 
\begin{equation*}
\pr(g)=r_{v_1}\cdots r_{v_k},	
\end{equation*}
where each $r_{v_i}\in \Orth(V)$ is the reflection in the hyperplane orthogonal to $Fv_i$, i.e., $r_{v_i} v_i=-v_i$ and $r_{v_i}x=x$ for every $x\in (Fv_i)^\perp$. 
In particular, $\GSpin(V)$ is given by products in \eqref{eq-GPin-homogeneous} with $k$ even.

Since $\GPin(V)$ consists of products in \eqref{eq-GPin-homogeneous}, the restriction of the Clifford norm $N$ to $\GPin(V)$ has its image in $F^\times$, thus we have a homomorphism
\begin{equation*}
N:\GPin(V)\to F^\times , \quad g\mapsto g\overline{g}.	
\end{equation*}
For any $z\in \ker(\pr)=F^\times$, we have $N(z)=z^2$. Also, for any $g\in \GPin(V)$, we have $g^{-1}=\frac{1}{N(g)}\overline{g}$. In particular, for any $g\in \GSpin(V)$, we have $g^{-1}=\frac{1}{N(g)}\overline{g}=\frac{1}{g g^*}g^*$.

We now describe the centers of $\GPin(V)$ and $\GSpin(V)$. For this purpose, take an orthogonal basis $\{b_1, \cdots, b_n\}$ of $V$ and denote $\zeta=b_1\cdots b_n$. 
\begin{lemma}\cite[Lemma 2.7]{Shimura2004}
\label{lemma-xi}
 The element $\zeta$ enjoys the following properties.
\begin{enumerate}
\item $\alpha(\zeta)v\zeta^{-1}=-v$ for all $v\in V$ and hence $\pr(\zeta)=-1$.
\item We have $\zeta^2=(-1)^{\frac{n(n-1)}{2}} b_1^2 b_2^2\cdots b_n^2$ and $\zeta^*=(-1)^{\frac{n(n-1)}{2}} \zeta$.
\item $F\zeta$ is independent of the choice of the orthogonal basis $\{b_1, \cdots, b_n\}$. 
\item $F+F\zeta$ is the center of $C(V)$ or of $C^+(V)$ according as $n$ is odd or even, respectively. 
\end{enumerate}
\end{lemma}

It is clear that $\zeta\in \GSpin(V)$ if $\dim(V)$ is even, and $\zeta\in \GPin(V)\backslash  \GSpin(V)$ if $\dim(V)$ is odd. The element $\zeta$ enjoys the following properties.

\begin{lemma}\cite[Theorem 3.6]{Shimura2004}
\label{lemma-center}
Let $Z_{\GPin(V)}$ and $Z_{\GSpin(V)}$ denote the centers of $\GPin(V)$ and $\GSpin(V)$ respectively. 	
\begin{enumerate}
\item If $\dim(V)$ is even and $\dim(V)>2$, then 
\begin{equation*}
Z_{\GPin(V)}=F^\times, \quad Z_{\GSpin(V)}=F^\times \cup F^\times \zeta.	
\end{equation*}
If $\dim(V)=2$, then 
\begin{equation*}
	Z_{\GPin(V)}=F^\times, \quad Z_{\GSpin(V)}=\GSpin(V).
\end{equation*}

\item If $\dim(V)$ is odd, then 
\begin{equation*}
Z_{\GPin(V)}=F^\times \cup F^\times \zeta, \quad Z_{\GSpin(V)}=F^\times. 
\end{equation*}
\end{enumerate}
In particular, $\ker(\pr)$ is the connected component of the center of $\GPin(V)$ as well as $\GSpin(V)$. 
\end{lemma}

Note that if $W\subset V$ is a non-degenerate subspace, then we have the inclusions of Clifford algebra $C(W)\subset C(V)$ and of even Clifford algebra $C^+(W)\subset C^+(V)$, which map $F^\times$ onto itself. These
induce the inclusions
 $\GPin(W)\subset \GPin(V)$ and $\GSpin(W)\subset \GSpin(V)$.

Now let $V=V_1\oplus V_2$ be an orthogonal sum decomposition, where both $V_1, V_2$ are non-degenerate. If $g_1=v_1\cdots v_k\in \GPin(V_1)$ and $g_2=w_1\cdots w_m\in \GPin(V_2)$, by using \eqref{eq-CliffordAlgebra-anticommute} we see that 
\begin{equation*}
g_1 g_2 g_1^{-1}=(-1)^{km} g_2.	
\end{equation*}
In particular, if $g_1\in \GSpin(V_1)$ and $g_2\in \GSpin(V_2)$, then $g_1g_2=g_2g_1$. Thus we can form the direct product $\GSpin(V_1)\times \GSpin(V_2)$. We have a homomorphism $\GSpin(V_1)\times \GSpin(V_2)\to \GSpin(V)$ given by $(g_1, g_2)\mapsto g_1g_2$ for $g_1\in \GSpin(V_1), g_2\in \GSpin(V_2)$, which induces an injection
\begin{equation*}
(\GSpin(V_1)\times \GSpin(V_2))/ \{(z, z^{-1}): z\in \GL_1\} \to \GSpin(V).	
\end{equation*}
Moreover, the diagram
\[
\begin{tikzcd}
& \GSpin(V_1)\times \GSpin(V_2) \ar[r] \arrow[d, "\pr\times \pr"] & \mathrm{GSpin}(V) \ar[d, "\pr"]  \\
 & \SO(V_1)\times \SO(V_2) \arrow[r, "\iota_{V_1,V_2}"] & \SO(V) 
\end{tikzcd}
\]
is commutative, where $\iota_{V_1,V_2}$ is the block diagonal embedding.

\subsection{Parabolic subgroups}
In this subsection we recall an explicit realization of parabolic subgroups of $\GSpin(V)$ in terms of Clifford algebra from \cite{Pollack2018}.
Let $X\subset V$ be an isotropic subspace, and we consider the parabolic subgroup $P_X$ of $\GSpin(V)$ that stabilizes $X$, namely
\begin{equation*}
P_X=\{ g\in \GSpin(V): g X g^{-1}=X\}. 	
\end{equation*}
The unipotent radical $N_X$ of $P_X$ has a two-step filtration $N_X\supset N_X^\prime\supset 1$. 
Inside the Clifford algebra, its Lie algebra is given by $X^\perp \cdot X$, and the exponential map identifies this nilpotent Lie algebra with $N_X$. Moreover, this filtration identifies
\[
N_X'\cong \wedge^2 X,
\qquad
N_X/N_X'\cong (X^\perp/X)\otimes X.
\]
If $X$ is one-dimensional, then $N_X^\prime=1$, while if $\dim(V)$ is even and $\dim(X)=\dim(V)/2$, then $X^\perp=X$ and hence $N_X=N_X^\prime$. In all other  cases, we have both $N_X/N_X^\prime$ and $N_X^\prime$ are nontrivial. 
The image of $P_X$ under $\pr$, $\pr(P_X)$, is the corresponding parabolic subgroup of $\SO(V)$.

Associated to $X$, one can define an increasing filtration $\mathcal{W}^{X}_{\bullet}$ on $C(V)$ as follows. First we define an increasing filtration on $V$ by
\begin{equation*}
\mathcal{W}_{k}^{X}(V):=\begin{cases}
0 & \text{ if }k<-1,\\
X & \text{ if }k=-1,\\
X^\perp & \text{ if }k=0,\\
V &\text{ if }k\ge 1.
\end{cases}
\end{equation*}
Next we define
\begin{equation*}
\mathcal{W}_{k}^{X}(C(V)):=\{x\in C(V): x=\sum x_1 x_2 \cdots x_m, x_i\in \mathcal{W}_{\alpha_i}^{X}(V), \alpha_1+\alpha_2+\cdots \alpha_m\le k \}.
\end{equation*}

\begin{lemma}
\cite[Lemma 2.3]{Pollack2018}
The parabolic subgroup $P_X$ and its unipotent radical $N_X$ are given by
\begin{equation*}
P_X = 	\mathcal{W}_{0}^{X}(C(V))\cap \GSpin(V), \quad N_X=(1+ \mathcal{W}_{-1}^{X}(C(V))) \cap \GSpin(V).
\end{equation*}
\label{lemma-parabolic-via-Clifford}
\end{lemma}

Now suppose $Y$ is an isotropic subspace of $V$ dual to $X$, and we take an orthogonal sum decomposition $V=X\oplus W\oplus Y$. 
Let $M_X$ be the Levi subgroup of the parabolic subgroup $P_X$ of $\GSpin(V)$. 
Then the image in $\SO(V)$ of $M_X$ is identified with $\GL_F(X)\times \SO(W)$. 
To describe the lifting of the $\GL_F(X)$ factor to $\GSpin(V)$ explicitly, we take  $X=\Span\{e_1, \cdots, e_m\}$ and $Y=\Span\{e_{-m}, \cdots, e_{-1}\}$ such that $\langle e_i, e_{-j}\rangle =\delta_{i,j}$ for $1\le i, j\le m$. Let $\{w_1, \cdots, w_{\dim W}\}$ be a basis of $W$. Then we have the following basis of $V$:
\begin{equation}
\label{eq-V-basis}
	\{e_1, \cdots, e_m, w_1, \cdots, w_{\dim W}, e_{-m}, \cdots, e_{-1}\}.
\end{equation}
Let $J_m$ denote the $m\times m$ anti-diagonal matrix with entries 1 on the anti-diagonal and 0 everywhere else. 
\begin{lemma}
\cite[Proposition 4.8]{Shimura2004}
\label{lemma-Levi-lifting}
There exists an injective algebraic group homomorphism 
\begin{equation*}
	\ell_m:\GL_F(X)=\GL_m(F)\to P_X
\end{equation*}
with the following properties:
\begin{enumerate}
\item For any $g\in \GL_m(F)$, we have $\pr(\ell_m(g))=\begin{pmatrix} g & & \\ &I_W &\\ & & J_m {}^tg^{-1} J_m \end{pmatrix}\in \SO(V)$, where the matrix is written with respect to the ordered basis \eqref{eq-V-basis}.
\item For any $g\in \GL_m(F)$, we have $N(\ell_m(g))=\det(g)$. 
\item For any diagonal element $\diag(a_1, \cdots, a_m)\in \GL_m(F)$, we have $\ell_m(\diag(a_1, \cdots, a_m))=\prod_{i=1}^m (a_i e_ie_{-i}+e_{-i} e_i)$. 
\item If $1\le m^\prime <m$ and $g_0\in \GL_F(\Span\{e_1, \cdots, e_{m^\prime} \} )=\GL_{m^\prime}(F)$, let $g\in \GL_m(F)$ be the element such that $g$ acts as $g_0$ on $\Span\{e_1, \cdots, e_{m^\prime} \}$ and acts trivially on $\Span\{e_{m^\prime+1}, \cdots, e_{m} \}$. Then $\ell_m(g)=\ell_{m^\prime}(g_0)$.
\item The image $\ell_m(\GL_m(F))$ is an algebraic subgroup of $P_X$.
\end{enumerate}
Moreover, $\ell_m$ is uniquely determined, as an algebraic group homomorphism, by properties (1)-(3).
\end{lemma}

\begin{lemma}
\label{lemma-Levi-as-product}
Let $\ell_m$ be the map in Lemma~\ref{lemma-Levi-lifting}. Then $\ell_m(\GL_m(F))\cap \GSpin(W)=\{1\}$.	
\end{lemma}

\begin{proof}
Let $g\in \GL_m(F)$ and assume that $\ell_m(g)\in \GSpin(W)$. Since $\GSpin(W)$ acts on $X$ trivially, $\pr(\ell_m(g))$ acts on $X$ trivially. 	On the other hand, by Lemma~\ref{lemma-Levi-lifting}, $\pr(\ell_m(g))$ acts $X$ by $g$. Thus $g=I_m$ and $\ell_m(g)=1$.
\end{proof}

Let $Z_m\subset \GL_m(F)$ be the unipotent radical of the Borel subgroup $B_m=T_m Z_m$ of $\GL_m(F)$ stabilizing the flag
\begin{equation*}
\Span\{e_1\} \subset \Span\{e_1, e_2\} \subset \cdots \subset \Span\{e_1, e_2, \cdots, e_m\}.	
\end{equation*}
In the next lemma, we describe the lifting of the unipotent group $Z_m$.

\begin{lemma}
\label{lemma-Levi-Unipotent-lifting}
For $1\le i <j\le m$, let $E_{i,j}$ denote the $m\times m$ matrix with one in the $(i,j)$-th entry and zeros everywhere else, and let $u_{i,j}(\gamma)=I_m + \gamma  E_{i,j}$ be the unipotent matrix in $\GL_m(F)$ where $\gamma \in F$. Then $\ell_m(u_{i,j}(\gamma))=1+\gamma  e_i e_{-j}$. 
\end{lemma}

\begin{proof}
Set $x=\gamma e_i e_{-j}$.
Since $i\neq j$, we have $\langle e_i,e_{-j}\rangle=e_i e_{-j}+e_{-j} e_i=0$, and hence $(e_i e_{-j})^2=e_i (-e_i e_{-j})e_{-j}=0$. Therefore $(1+x)^{-1}=1-x.$ By the uniqueness of $\ell_m$ in Lemma~\ref{lemma-Levi-lifting}, it suffices to check that $1+x$ agrees with $u_{i,j}(\gamma)$ on properties (1)-(3) of Lemma~\ref{lemma-Levi-lifting}.

We first compute the action of $1+x$ on $V$ by conjugation. For $e_j$, using
$e_{-j} e_j=1-e_je_{-j}$, we get
\begin{equation*}
\begin{split}
(1+x)e_j(1-x)
=e_j+x e_j-e_jx-xe_jx  = e_j+\gamma e_i e_{-j} e_j-\gamma e_j e_i e_{-j}-\gamma^2 e_i e_{-j} e_j e_i e_{-j}.
\end{split}
\end{equation*}
Now $e_i e_{-j} e_j=e_i(1-e_je_{-j})=e_i-e_i e_j e_{-j}$, and $e_j e_i e_{-j}=-e_i e_j e_{-j}.$
Moreover, the last term is zero because $e_i e_{-j} e_je_i=-e_{-j} e_ie_j e_i=e_{-j} e_j e_i^2=0$. Hence $(1+x)e_j(1-x)=e_j+\gamma e_i.$
For $k\neq j$, a direct calculation gives $(1+x)e_k(1-x)=e_k.$
Thus the action of $1+x$ on $X=\Span\{e_1,\dots,e_m\}$ is given by
\[
e_j\mapsto e_j+\gamma e_i,\qquad e_k\mapsto e_k\quad \text{ for $k\neq j$}.
\]
 
Next, since every $w\in W$ is orthogonal to both $e_i$ and $e_{-j}$, it commutes
with the product $e_ie_{-j}$. Hence
\[
(1+x)w(1-x)=w.
\]

On the $Y$-part, one checks similarly that the action of $1+x$ is given by
\begin{equation*}
e_{-i}\mapsto e_{-i}-\gamma e_{-j}, \quad e_{-k}\mapsto e_{-k}  \quad \text{ for $k\neq i$}.
\end{equation*}
We conclude that the action of $1+x$ on $V$ is given by
\[
\begin{pmatrix}
u_{i,j}(\gamma) & & \\
& I_W & \\
& & J_m{}^tu_{i,j}(\gamma)^{-1}J_m
\end{pmatrix}.
\]

We now compute the Clifford norm of $1+x$. Since
$(e_ie_{-j})^*=e_{-j} e_i=-e_i e_{-j},$
we have $
(1+\gamma e_ie_{-j})^*=1-\gamma e_ie_{-j}.$
Thus
\[
N(1+\gamma e_ie_{-j})
=
(1+\gamma e_ie_{-j})(1-\gamma e_ie_{-j})
=
1-\gamma^2(e_ie_{-j})^2
=
1.
\]
This agrees with $\det(u_{i,j}(\gamma))=1.$

Finally, the unipotent matrix $u_{i,j}(\gamma)$ is diagonal if and only if $\gamma =0$, reducing it to the identity matrix $I_m=\diag(1, \cdots, 1)$, which agrees with $\prod_{i=1}^m (e_i e_{-i}+e_{-i} e_i)=1$. 

By the uniqueness of the lift $\ell_m$ in Lemma~\ref{lemma-Levi-lifting}, we obtain $\ell_m(u_{i,j}(\gamma))=1+\gamma e_i e_{-j}$.
\end{proof}

\begin{lemma}
\label{lemma-Levi-GL-GSpin-commute}
Let $\ell_m$ be as in Lemma~\ref{lemma-Levi-lifting}. Then $\GSpin(W)$ commutes with $\ell_m(\GL_F(X))$.	
\end{lemma}

\begin{proof}
Note that every $w\in W$ anti-commutes with each of the isotropic basis vectors $e_1, \cdots, e_m$, $e_{-m}, \cdots, e_{-1}$, so every $w\in W$ commutes with every quadratic monomial in $e_1, \cdots, e_m$, $e_{-m}, \cdots, e_{-1}$. By Lemma~\ref{lemma-Levi-lifting} (3), every $w\in W$ commutes with the lifts of diagonal matrices $\ell_m(\diag(a_1, \cdots, a_m))$. By Lemma~\ref{lemma-Levi-Unipotent-lifting}, every $w\in W$ commutes with the lifts of upper triangular matrices $I_m+\gamma E_{i,j}$ with $1\le i<j\le m$. By a similar computation as in Lemma~\ref{lemma-Levi-Unipotent-lifting} for the lifts of lower triangular unipotent matrices, every $w\in W$ commutes with the lifts of lower triangular matrices $I_m+\gamma E_{i,j}$ with $1\le j<i\le m$. Hence every $w\in W$ commutes with $\ell_m(\GL_F(X))$. Every element in $\GSpin(W)$ can be written as a product $h=w_1 \cdots w_k$ where $w_1, \cdots, w_k\in W$ are non-isotropic and $k$ is even; see \eqref{eq-GPin-homogeneous}. Hence $\GSpin(W)$ commutes with $\ell_m(\GL_F(X))$.
\end{proof}

\section{Preparation of the proof}
\label{section-preparation-of-proof}
\subsection{The Bessel subgroup $H$ and the character $\xi$}
\label{subsection-Bessel-subgroup}

Let $V$ be a non-degenerate quadratic space over $F$, and assume that
$\dim(V)>3$; see Remark~\ref{Intro-remark-dim>3}. Let $V_0\subset V$ be a non-degenerate subspace such that $V_0^\perp$ is split of dimension $2r+1$.
Write $V_0^\perp=X_r\oplus Y_r\oplus Fe_0,$
where $X_r$ and $Y_r$ are maximal isotropic subspaces of $V_0^\perp$ in duality,
and where $e_0$ is non-isotropic. Set $V_0'=V_0\oplus Fe_0$.
Thus $V=X_r\oplus V_0'\oplus Y_r.$

Let $P_r=N_{P_r}\rtimes M_{P_r}$ be the parabolic subgroup of $\GSpin(V)$
stabilizing $X_r$, namely
\[
P_r=\{g\in \GSpin(V):gX_rg^{-1}=X_r\}.
\]
Then the Levi subgroup $M_{P_r}=\ell_r(\GL_F(X_r))\times \GSpin(V_0^\prime)\cong   \GL_F(X_r) \times \GSpin(V_0^\prime)$, where we recall $\ell_r$ is the map from Lemma~\ref{lemma-Levi-lifting}. The unipotent radical $N_{P_r}$ has a two-step filtration $N_{P_r}\supset N_{P_r}'\supset 1,$ where $N_{P_r}'\cong \wedge^2 X_r$ and $N_{P_r}/N_{P_r}'\cong X_r\otimes V_0'.$ We have a short exact sequence
\[
\begin{tikzcd}
1 \ar[r] &
N_{P_r}^\prime \ar[r] &
N_{P_r} \ar[r] &
N_{P_r}/ N_{P_r}^\prime \ar[r] &
1.
\end{tikzcd}
\]

Let $l_{X_r}:X_r\to F$ be a non-zero $F$-linear homomorphism, and let
$l_{V_0}:V_0'\to F$ be a non-zero $F$-linear homomorphism which vanishes on $V_0$. Then $l_{V_0}$ is supported on the line $Fe_0$. We obtain a linear map $l_{X_r}\otimes l_{V_0}:X_r\otimes V_0'\longrightarrow F.$ Explicitly, we have
\begin{equation*}
	(l_{X_r}\otimes l_{V_0})(x\otimes (v+c e_0))=c\cdot  l_{X_r}(x) l_{V_0}(e_0)\quad \text{ for }x\in X_r, v\in V_0, c\in F.
\end{equation*}
By composing with the natural quotient map, we define a map $l_{N_{P_r}}:N_{P_r}\longrightarrow F$ by
\[
\begin{tikzcd}
l_{N_{P_r}}:N_{P_r} \ar[r] &
N_{P_r}/N_{P_r}' \ar[r, "\sim"] &
X_r\otimes V_0' \ar[r, "l_{X_r}\otimes l_{V_0}"] &
F.
\end{tikzcd}
\]

Let $Z_{X_r}$ be a maximal unipotent subgroup of $\GL_F(X_r)$ stabilizing
$l_{X_r}$. Then the subgroup $Z_{X_r}\times \GSpin(V_0)$
of the standard Levi subgroup $M_{P_r}$ fixes $l_{N_{P_r}}$. Indeed, $Z_{X_r}$ fixes
$l_{X_r}$ by definition, and $\GSpin(V_0)$ acts trivially on $Fe_0$.

Set
\[
N=N_{P_r}\rtimes Z_{X_r}\  \subset \GSpin(V).
\]
We define the $r$-th Bessel subgroup $H$ of $\GSpin(V)$ by
\[
H= N\rtimes \GSpin(V_0).
\]
We extend $l_{N_{P_r}}$ to $N=N_{P_r}\rtimes Z_{X_r}$ by making it trivial on
$Z_{X_r}$. Let $\psi$ be a non-trivial additive unitary character of $F$, and
let
\[
\lambda_{X_r}:Z_{X_r}\to \mathbb{S}^1
\]
be a generic unitary character. We define a character $\xi$ of $H$ by
\[
\xi(nzh)=\lambda_{X_r}(z)\psi(l_{N_{P_r}}(n)),
\]
for $z\in Z_{X_r}$, $n\in N_{P_r}$, and $h\in \GSpin(V_0)$. This is well-defined
because $Z_{X_r}$ and $\GSpin(V_0)$ fix $l_{N_{P_r}}$. The pair $(H,\xi)$ is uniquely determined, up to conjugacy in
$G=\GSpin(V)\times \GSpin(V_0)$, by the pair $V_0\subset V$.

We now give a more explicit description of $(H,\xi)$. Let
$\{e_1,\dots,e_r\}$ be a basis of $X_r$, and let
$\{e_{-r},\dots,e_{-1}\}$ be the dual basis of $Y_r$, with $\langle e_i,e_{-j}\rangle=\delta_{ij}$ for $1\le i,j\le r.$
Thus, in the Clifford algebra, we have 
\begin{equation}
\label{eq-relation-eiej}
e_i^2=e_{-i}^2=0,\qquad
e_i e_{-i}+e_{-i}e_i=1,\qquad
e_i e_{-j}=-e_{-j}e_i
\end{equation}
for $1\le i,j\le r$ with $i\neq j$.

Let $Q_r=N_{Q_r}\rtimes M_{Q_r}$ denote the standard parabolic subgroup of
$\GSpin(V)$ stabilizing the flag
\begin{equation}
\label{eq-flag-Xr}	
0=X_0\subset X_1=Fe_1\subset X_2=Fe_1\oplus Fe_2
\subset\cdots\subset X_r=Fe_1\oplus\cdots\oplus Fe_r.
\end{equation}
Then the Levi subgroup $M_{Q_r} =\ell_r((\GL_1)^r)\times \GSpin(V_0^\prime) \cong (\GL_1)^r\times \GSpin(V_0^\prime)$, where $(\GL_1)^r\subset \GL_F(X_r)$ is the diagonal torus with respect to the basis $\{e_1, \cdots, e_r\}$.
The $(\GL_1)^r$ factor is represented explicitly by 
$\prod_{i=1}^r(a_i e_i e_{-i}+e_{-i}e_i)\in \GSpin(V),$ where $a_i\in F^\times$.
It acts by
\begin{equation}
\label{eq-GL-action-on-X}
\prod_{i=1}^r(a_i e_i e_{-i}+e_{-i}e_i)  e_k \prod_{i=1}^r(a_i e_i e_{-i}+e_{-i}e_i) ^{-1}=a_k e_k,
\qquad 1\le k\le r,
\end{equation}
and
\begin{equation}
\label{eq-GL-action-on-Y}
\prod_{i=1}^r(a_i e_i e_{-i}+e_{-i}e_i) e_{-k} \prod_{i=1}^r(a_i e_i e_{-i}+e_{-i}e_i) ^{-1}=a_k^{-1}e_{-k},
\qquad 1\le k\le r.
\end{equation}
The torus element $\prod_{i=1}^r(a_i e_i e_{-i}+e_{-i}e_i)$ fixes the line $Fe_0$. Moreover, $\GSpin(V_0')$ acts trivially on $X_r\oplus Y_r$.
The unipotent radical of $Q_r$ is $N_{Q_r}=Z_r\ltimes N_{P_r},$
where $Z_r=Z_{X_r}$ is the unipotent radical of the Borel subgroup
$B_r=T_rZ_r$ of $\GL_r=\GL_F(X_r)$ stabilizing the flag \eqref{eq-flag-Xr}.
Thus, with this choice of $Z_r$, we have $N_{Q_r}=N.$

Define a map
\[
l_r:N_{Q_r}\longrightarrow F^r
\]
by
\[
l_r(u)=(x_1,\dots,x_{r-1},z),
\]
where
\begin{equation}
\label{eq-char-xi-x}
x_i=x_i(u)=\langle \pr(u)e_{i+1},e_{-i}\rangle,
\qquad 1\le i\le r-1,
\end{equation}
and
\begin{equation}
\label{eq-char-xi-z}
z=z(u)=\langle \pr(u)e_0,e_{-r}\rangle.
\end{equation}

We first observe that $l_r$ is fixed by $\GSpin(V_0)$-conjugation. Indeed, for
$h\in \GSpin(V_0)$ and $u\in N_{Q_r}$, since $\pr(h)$ fixes each $e_i$,
$e_{-i}$, and $e_0$, we have
\[
\begin{split}
\langle \pr(huh^{-1})e_{i+1},e_{-i}\rangle
&=
\langle \pr(h)\pr(u)\pr(h^{-1})e_{i+1},e_{-i}\rangle  \\
&=
\langle \pr(h)\pr(u)e_{i+1},e_{-i}\rangle  \\
&=
\langle \pr(u)e_{i+1},\pr(h^{-1})e_{-i}\rangle  \\
&=
\langle \pr(u)e_{i+1},e_{-i}\rangle.
\end{split}
\]
Similarly,
\[
\langle \pr(huh^{-1})e_0,e_{-r}\rangle
=
\langle \pr(u)e_0,e_{-r}\rangle.
\]
Hence $l_r$ is invariant under $\GSpin(V_0)$-conjugation. The conjugation action of the torus
$t=\prod_{i=1}^r(a_i e_i e_{-i}+e_{-i}e_i)$, where $a_i\in F^\times$, is given by
\[
\begin{split}
\langle \pr(tut^{-1})e_{i+1},e_{-i}\rangle
&=
\langle \pr(t)\pr(u)\pr(t^{-1})e_{i+1},e_{-i}\rangle \\
&=
a_{i+1}^{-1}
\langle \pr(t)\pr(u)e_{i+1},e_{-i}\rangle \quad (\text{by \eqref{eq-GL-action-on-X}})  \\
&= a_{i+1}^{-1}
\langle \pr(u)e_{i+1},\pr(t^{-1})e_{-i}\rangle \\
&= a_i a_{i+1}^{-1}
\langle \pr(u)e_{i+1},e_{-i}\rangle \quad (\text{by \eqref{eq-GL-action-on-Y}}).
\end{split}
\]
Similarly, since $t$ fixes $e_0$ and $t^{-1}e_{-r}t=a_r e_{-r}$, we have $\langle \pr(tut^{-1})e_0,e_{-r}\rangle=a_r \langle \pr(u) e_0,e_{-r}\rangle$.
The subgroup of $M_{Q_r}$ fixing $l_r$ is precisely $\GSpin(V_0)$.

For a non-trivial additive unitary character $\psi$ of $F$, define a character
$\xi$ of $N_{Q_r}$ by
\begin{equation}
\label{eq-xi-defn1}
\xi(u)=\psi\left(\sum_{i=1}^{r-1}x_i+z\right),
\end{equation}
where $x_i$ and $z$ are given by \eqref{eq-char-xi-x} and \eqref{eq-char-xi-z}. Since $l_r$ is invariant under $\GSpin(V_0)$-conjugation, $\xi$ is invariant
under $\GSpin(V_0)$-conjugation. Hence $\xi$ extends to a character
\[
\xi:N_{Q_r}\rtimes \GSpin(V_0)\longrightarrow \mathbb{C}^\times
\]
by $\xi(nh)=\xi(n)$ for $n\in N_{Q_r},\ h\in \GSpin(V_0).$ As noted above, up to conjugacy in
$G=\GSpin(V)\times \GSpin(V_0)$, the pair $(H, \xi)$ depends only on the pair $V_0\subset V$, and not on the choices of $\psi$, or the basis vectors $e_i$.

We now prove that the Bessel subgroup $H$ is unimodular.

\begin{lemma}
\label{lemma-H-unimodular}
The group $H$ is unimodular.
\end{lemma}

\begin{proof}
Recall that $H=N_{Q_r}\rtimes \GSpin(V_0),$ $N_{Q_r}=Z_r\ltimes N_{P_r}.$
The group $N_{Q_r}$ is unipotent, and hence is unimodular. The group $\GSpin(V_0)$
is reductive, and hence is also unimodular. Therefore it remains to check that
the conjugation action of $\GSpin(V_0)$ on $N_{Q_r}$ has determinant one.

The group $\GSpin(V_0)$ acts trivially on the factor $Z_r$, since $Z_r$ is
contained in the lifted $\GL_F(X_r)$-factor and $\GSpin(V_0)$ acts trivially on
$X_r\oplus Y_r$. Thus it suffices to analyze its action on $N_{P_r}$.

At the level of Lie algebras, the two-step filtration of $N_{P_r}$ gives $\Lie(N_{P_r})
\simeq (X_r\otimes V_0')\oplus \wedge^2 X_r$ as a vector space. The subgroup $\GSpin(V_0)$ acts trivially on $X_r$, on
$\wedge^2X_r$, and on the line $Fe_0\subset V_0'$. On the $V_0$-summand of
$V_0'=V_0\oplus Fe_0$, it acts through $\pr:\GSpin(V_0)\longrightarrow \SO(V_0).$ Hence for $h\in \GSpin(V_0)$, the determinant of its action on
$\Lie(N_{P_r})$ is $\det(\pr(h))^{\dim X_r}=1,$
because $\pr(h)\in \SO(V_0)$. Therefore the modular character of the semidirect product $N_{Q_r}\rtimes \GSpin(V_0)$ is trivial. Hence $H$ is unimodular.
\end{proof}

\subsection{The group $\GSpin(V^\prime)$}
\label{subsection-Group-GSpinVprime}
Consider $V^\prime=V\oplus F f_0$. We equip $V^\prime$ with a symmetric bilinear form $\langle \cdot, \cdot \rangle_{V^\prime}$ so that 
\begin{equation*}
\langle v_1, v_2 \rangle_{V^\prime}=\langle v_1, v_2\rangle, \quad \langle v, f_0\rangle_{V^\prime} =0, \quad \langle f_0, f_0\rangle_{V^\prime}=-\langle e_0, e_0\rangle, \quad \text{ for all }v, v_1, v_2\in V. 
\end{equation*}
Denote $e_{r+1}=e_0+f_0$, $e_{-(r+1)}=\frac{1}{2\langle e_0, e_0\rangle }(e_0-f_0)$. 
Then $V^\prime$ is the orthogonal sum 
\begin{equation*}
V^\prime= (X^\prime_{r+1}\oplus Y^\prime_{r+1}) \oplus V_0
\end{equation*}
where 
\begin{equation*}
X_{r+1}^\prime= X_r\oplus F e_{r+1}, \quad Y^\prime_{r+1}= Y_r \oplus F e_{-(r+1)}
\end{equation*}
are totally isotropic subspaces.  

Recall $V_0^\prime=V_0\oplus F e_0$, and $P_r=N_{P_r} \rtimes M_{P_r}$ is the parabolic subgroup of $\GSpin(V)$ stabilizing $X_r$, where $M_{P_r}\cong \GSpin(V_0^\prime)\times \GL_r$, $\GL_r=\GL_F(X_r)$.
Denote by $P_{r+1}^\prime= N_{P_{r+1}^\prime} \rtimes M_{P_{r+1}^\prime}$ the parabolic subgroup of $\GSpin(V^\prime)$ stabilizing $X_{r+1}^\prime$. Then we have
\begin{equation*}
\begin{split}
M_{P_{r+1}^\prime} \cong   \GSpin(V_0)\times \GL_{r+1}, \quad  \GL_{r+1}:=\GL_F(X_{r+1}^\prime) \supset \GL_r.
\end{split}	
\end{equation*}
We view $\GL_r$ as a subgroup of $\GL_{r+1}$ by acting trivially on $e_{r+1}$.
Recall that $Z_r$ is the unipotent radical  of the Borel subgroup $B_{r}=T_{r}Z_{r}$  of $\GL_F(X_{r} )$ stabilizing the flag \eqref{eq-flag-Xr}.
Let $Z_{r+1}$ be the unipotent radical of the Borel subgroup $B_{r+1}=T_{r+1}Z_{r+1}$ of $\GL(X_{r+1}^\prime)$ stabilizing the flag
\begin{equation}
\label{eq-flag-Xr+1}
 X_1 \subset X_2 \subset\cdots \subset 	  X_r \subset X_{r+1}^\prime.
\end{equation}
We observe that 
\begin{equation}
\label{eq-Bessel-subgp}
H=N_{P_r}\rtimes (\GSpin(V_0)\times Z_r)\subset P_r=	N_{P_r} \rtimes (\GSpin(V_0^\prime)\times \GL_r)
\end{equation}
and
\begin{equation*}
H \subset 	P_{r+1}^\prime\cap \GSpin(V).
\end{equation*}

Now we analyze the structure of the group $P_{r+1}^\prime\cap \GSpin(V)$. Let $P^\prime_r=N_{P_{r}^\prime} \rtimes M_{P_{r}^\prime}$ be the parabolic subgroup of $\GSpin(V^\prime)$ stabilizing $X_r$.

\begin{lemma}
We have $P_{r+1}^\prime\cap \GSpin(V) =(N_{P_{r}^\prime}\cap \GSpin(V)) \rtimes (\GL(X_r)\times \GSpin(V_0))   \subset P_r^\prime\cap \GSpin(V)$.	
\end{lemma}

\begin{proof}
We first show that $P_{r+1}^\prime\cap \GSpin(V)  \subset P_r^\prime\cap \GSpin(V)$.	 Take any $g\in P_{r+1}^\prime\cap \GSpin(V)$. 
It suffices to show that $\pr(g)$ also stabilizes $X_r$, i.e., for any $x\in X_r$, we need to show that $\langle \pr(g)x, e_0-f_0 \rangle_{V^\prime}=0$. We compute that
\begin{equation*}
\begin{split}
	\langle \pr(g)x, e_0-f_0 \rangle_{V^\prime} &= \langle x, \pr(g)^{-1}(e_0-f_0) \rangle_{V^\prime} \\
	&=  \langle x, \pr(g)^{-1}(e_{r+1}-2f_0) \rangle_{V^\prime} \\
	&= \langle x, \pr(g)^{-1}e_{r+1} \rangle_{V^\prime} -2\langle x, \pr(g)^{-1} f_0 \rangle_{V^\prime} .
\end{split}
\end{equation*}
Since $g^{-1}\in \GSpin(V)$, $g^{-1}$ acts trivially on $f_0$, i.e., $\pr(g^{-1})f_0=f_0$. Hence
\begin{equation*}
	\langle x, \pr(g)^{-1} f_0 \rangle_{V^\prime}=\langle x,  f_0 \rangle_{V^\prime}=0.
\end{equation*}
Since $g^{-1}\in P_{r+1}^\prime$, it stabilizes $X_{r+1}^\prime$ and $\pr(g^{-1})e_{r+1}\in X_{r+1}^\prime$. Since $X_{r+1}^\prime$ is totally isotropic, 
we have
\begin{equation*}
\langle x, \pr(g)^{-1}e_{r+1}\rangle_{V^\prime}=0.	
\end{equation*}
Thus, 
\begin{equation*}
\langle \pr(g)x, e_0-f_0 \rangle_{V^\prime}=0.	
\end{equation*}
This proves $P_{r+1}^\prime\cap \GSpin(V)  \subset P_r^\prime\cap \GSpin(V)$.

Next, we show that as a subgroup of $P_r^\prime\cap \GSpin(V)$, 
$$
P_{r+1}^\prime\cap \GSpin(V) =(N_{P_{r}^\prime}\cap \GSpin(V)) \rtimes (\GL(X_r)\times \GSpin(V_0)).
$$ 
Let $h$ be an element of $P_r^\prime\cap \GSpin(V)$, so that $\pr(h)x\in X_{r}$ for all $x\in X_r$ and $\pr(h)f_0=f_0$. Note that
\begin{equation*}
\begin{split}
h\in P_{r+1}^\prime\cap \GSpin(V) &\iff \pr(h) e_{r+1}\in X_{r+1}^\prime \\
&\iff \pr(h)e_0+f_0\in X_{r+1}^\prime\\
&\iff \pr(h)e_0-e_0\in X_{r+1}^\prime.
\end{split}
\end{equation*}
Also, if $x\in X_r$, then $\pr(h)^{-1}x\in X_r$ and hence
\begin{equation*}
\langle \pr(h)e_0, x\rangle_{V^\prime}= \langle e_0, \pr(h)^{-1}x \rangle_{V^\prime}=0.
\end{equation*}
Thus we may write $\pr(h)e_0=\lambda e_0+w+x$ for some $\lambda\in F$, $w\in V_0$, $x\in X_r$. The condition $\pr(h)e_0-e_0\in X_{r+1}^\prime$ if and only if $\lambda=1$ and $w=0$, implying that $\pr(h)$ fixes $e_0$ modulo $X_r$ and hence stabilizes $V_0$ modulo $X_r$, consequently, $\pr(h)\in (N_{P_{r}^\prime}\cap \GSpin(V)) \rtimes (\GL(X_r)\times \GSpin(V_0))$, as desired. This completes the proof. 
\end{proof}

Let $\alpha:P_{r+1}^\prime \twoheadrightarrow   \ell_{r+1}(\GL_{r+1})\times \GSpin(V_0)\cong \GL_{r+1}\times \GSpin(V_0)$ be the natural projection, whose kernel is the unipotent radical $N_{P_{r+1}^\prime}$.

\begin{lemma}
\label{lemma-commutative-diagram}
The	natural projection map $\alpha:P_{r+1}^\prime \twoheadrightarrow\GSpin(V_0)\times \GL_{r+1}$ induces the following commutative diagram with exact rows, where the vertical arrows are inclusions:
\[  \begin{CD}
1 @>>> N_{P_{r+1}^\prime} @>>> P_{r+1}^\prime @>>> \GL_{r+1} \times \GSpin(V_0) @>>> 1 \\
&  &  @AAA     @AAA  @AAA  \\
1 @>>> N_{P_{r+1}^\prime}\cap \GSpin(V)   @>>> P_{r+1}^\prime\cap \GSpin(V) @>>>  R \times \GSpin(V_0) @>>> 1 \\
&  &@|  @AAA     @AAA     \\
1 @>>> N_{P_{r+1}^\prime}\cap \GSpin(V)  @>>> N_{P_r^\prime}\cap\GSpin(V) @>>>  \Hom(F e_{r+1}, X_r) @>>> 1.
\end{CD} \]
Here, $R\subset \GL(X_{r+1}^\prime)$ is the mirabolic subgroup which stabilizes the co-dimension one subspace $X_r\subset X_{r+1}^\prime$ and fixes $e_{r+1}$ modulo $X_r$, and $\Hom(F e_{r+1}, X_r)$ is the unipotent radical of $R$.
\end{lemma}

\begin{proof}
We consider the restriction of $\alpha$ to the subgroup $P_{r+1}^\prime\cap \GSpin(V) =N_{P_{r}} \rtimes (\GL(X_r)\times \GSpin(V_0))$. Observe that the subgroup $\GL_r\times \GSpin(V_0)$ is mapped isomorphically to its image in $\GL_{r+1} \times \GSpin(V_0)$, and its image is precisely $\GL_r\times \GSpin(V_0)$, which is the Levi subgroup of $R\times \GSpin(V_0)$. 

Next, we analyze the kernel of $\alpha$ restricted to $P_{r+1}^\prime\cap \GSpin(V)$, which is given by
$$\ker(\alpha|_{P_{r+1}^\prime\cap \GSpin(V)}) =N_{P_{r+1}^\prime}\cap \GSpin(V).
$$ 
We claim that $\ker(\alpha|_{P_{r+1}^\prime\cap \GSpin(V)})$ is contained in $N_{P_r^\prime}\cap \GSpin(V)$. To prove this, take $h\in N_{P_{r+1}^\prime}\cap \GSpin(V)$, so that $\pr(h)f_0=f_0$. Our goal is to show that $\pr(h)$ acts trivially on $X_r$, and acts trivially on $V_0^\prime$ modulo $X_r$. Since $h\in  N_{P_{r+1}^\prime}$, $\pr(h)$ acts trivially on $X_{r+1}^\prime$, and acts trivially on $V_0$ modulo $X_{r+1}^\prime$. Since $X_r\subset X_{r+1}^\prime$, $\pr(h)$ acts trivially on $X_r$. Also, for $w\in V_0$, we have $\pr(h)w-w\in X_{r+1}^\prime\cap V=X_r$. It remains to show that $\pr(h)e_0-e_0\in X_r=X_{r+1}^\prime\cap V$. Since $\pr(h)e_0-e_0\in V$, it suffices to show that $\pr(h)e_0-e_0\in X_{r+1}^\prime$. We have
\begin{equation*}
\pr(h)e_0-e_0=\pr(h)(e_0+f_0)-(e_0+f_0)=\pr(h)e_{r+1}-e_{r+1}	\in X_{r+1}^\prime,
\end{equation*}
as desired. This proves the claim. 

Now we show that the map $\alpha$ induces an isomorphism
\begin{equation*}
	(N_{P_r^\prime}\cap\GSpin(V)) /  (N_{P_{r+1}^\prime}\cap \GSpin(V)) \cong  \Hom(F e_{r+1}, X_r).
\end{equation*}
If $h\in N_{P_r^\prime}\cap\GSpin(V)$, then $\pr(h)$ acts trivially on $X_r$, and $\pr(h)$ acts trivially on $V_0^\prime$ modulo $X_r$, and $\pr(h)f_0=f_0$. Then
\begin{equation*}
\pr(h)e_{r+1}-e_{r+1}=\pr(h)(e_0+f_0)-(e_0+f_0)=\pr(h)e_0-e_0\in X_r.	
\end{equation*}
It follows that $\alpha(h)$ lies in the unipotent radical $\Hom(Fe_{r+1}, X_r)$ of $R$. 
It remains to prove that $\alpha|_{P_{r+1}^\prime\cap \GSpin(V)}$ is surjective onto $\Hom(Fe_{r+1}, X_r)$. 
For any $x\in X_r$, set $h_x=1+\frac{1}{\langle e_0,e_0\rangle}xe_0.$ Then $h_x\in N_{P_r^\prime}\cap \GSpin(V)$, fixes $X_r$ and $V_0$, and satisfies $\pr(h_x)e_0=e_0+x$. Then $\pr(\alpha(h_x))(e_{r+1})=x$, as desired. Thus we have the commutative diagram. 
\end{proof}

By Lemma~\ref{lemma-commutative-diagram}, we see that the preimage of $Z_{r+1}\times\GSpin(V_0)$ in $P_{r+1}^\prime\cap \GSpin(V)$ under the restriction of $\alpha$ is precisely the Bessel subgroup 
\begin{equation*}
H=N_{P_r}\rtimes (\GSpin(V_0)\times Z_r).
\end{equation*}
Thus the restriction of $\alpha$ gives a surjective homomorphism
\begin{equation*}
\alpha|_H:H	 \twoheadrightarrow Z_{r+1}\times\GSpin(V_0).
\end{equation*}
Moreover, the character $\xi:H\to \mathbb{C}^\times$ is trivial on the kernel of $\alpha|_H$.  Indeed, the kernel of $\alpha|_H$ is $N_{P_{r+1}'}\cap \GSpin(V)$, and this subgroup is contained in the kernel of $\xi$.
Hence $\xi$ descends to a character on $Z_{r+1}\times\GSpin(V_0)$. Equivalently, there exists a character $\chi$ of $Z_{r+1}\times \GSpin(V_0)$ such that 
\begin{equation}
\label{eq-xi-chi}
\xi=\chi\circ 	\alpha|_{H}.
\end{equation}
Here, $\chi$ is a unitary generic character on $Z_{r+1}$, extended to a character on $Z_{r+1}\times\GSpin(V_0)$ by $\chi(uh)=\chi(u)$ for $u\in Z_{r+1}, h\in \GSpin(V_0)$. 

\begin{remark}
\label{remark-chi-xi-compatibility}
The character $\chi$ is compatible with the character $\xi$ defined in Section~\ref{subsection-Bessel-subgroup} in the following sense.
Recall that on $N_{Q_r}=Z_r\ltimes N_{P_r}$, $\xi$ was defined as $\xi(u)=\psi(x_1+\cdots+x_{r-1}+z)$, where $x_i$ and $z$ are given by \eqref{eq-char-xi-x} and \eqref{eq-char-xi-z}. Under the quotient map $\alpha|_H:H\twoheadrightarrow Z_{r+1}\times \GSpin(V_0)$, the coordinates $x_1,\dots,x_{r-1}$ correspond to the simple root coordinates
of the subgroup $Z_r\subset Z_{r+1}$, while the coordinate $z$ corresponds to
the additional simple root subgroup $\Hom(Fe_{r+1},X_r)$ appearing in the unipotent radical of the mirabolic subgroup $R\subset \GL_{r+1}=\GL(X_{r+1}^\prime)$. Consequently, up to conjugation, the character $\chi$ of $Z_{r+1}$ is the generic
character obtained from the same linear functional $(x_1,\dots,x_{r-1},z)\longmapsto x_1+\cdots+x_{r-1}+z.$
\end{remark}

 \subsection{Induced representations of $\GSpin(V')$}
\label{subsection-induced-rep}

Let $\pi_0$ be an irreducible admissible representation of $\GSpin(V_0)$ and let 
$\sigma$ be an irreducible admissible $\chi^{-1}$-generic representation of $\GL_{r+1}$. 
The existence of such $\sigma$ is well known; see for example \cite[Theorem 9.7]{Zelevinsky1980} when $F$ is non-archimedean and \cite[Theorem 9.1]{CasselmanHechtMilivciv2000} when $F$ is archimedean. 
Put 
\begin{equation*}
\rho:= \pi_0\otimes \sigma,
\end{equation*}
which is an admissible representation of the Levi subgroup $M_{P_{r+1}'}\cong \GSpin(V_0)\times \GL_{r+1}.$
For $s\in \mathbb C$, define $\rho_s(h m)=|\det(m)|^s\rho(h m),$ for $h\in \GSpin(V_0),\ m\in \GL_{r+1}.$ 
\textbf{Throughout, we identify $\GL_{r+1}$ with its lift $\ell_{r+1}(\GL_{r+1})$ in $\GSpin(V^\prime)$, where $\ell_{r+1}$ is given in Lemma~\ref{lemma-Levi-lifting}.}

Let $\pi_s'=\Ind_{P_{r+1}'}^{\GSpin(V')}(\rho_s)$ denote the (unnormalized) induced representation  of $\GSpin(V^\prime)$, consisting of all smooth functions $f:\GSpin(V^\prime)\to \rho$ satisfying
\begin{equation}
\label{eq-induced-repn-equivariance}
f( ugm x)	=|\det(m)|^{s} \rho(g m)(f(x))
\end{equation}
for all $u\in N_{P_{r+1}^\prime}$, $g\in \GSpin(V_0)$, $m\in \GL_{r+1}$, $x\in \GSpin(V^\prime)$. 
Note that for $z\in F^\times\subset \GSpin(V_0)$, we have 
\begin{equation}
\label{eq-induced-repn-f(zx)}
f(zx)=\rho(z)f(x)=\omega_{\pi_0}(z)f(x), \quad  \forall x\in \GSpin(V^\prime).	
\end{equation}

We have the following lemma, which will be used in Section~\ref{subsection-proof-of-thm}.

\begin{lemma}
\label{lemma-irreducible}
The representation $\pi_s^\prime$ is irreducible except for a measure zero subset of $\mathbb{C}$. 
\end{lemma}

\begin{proof}
When $F$ is archimedean, this is a consequence of \cite[Theorem 1.1]{SpehVogan1980}. When $F$ is non-archimedean, this follows from Bernstein's generic irreducible theorem for parabolic induction (see \cite[Proposition VI.8.4]{Renard2010}).
\end{proof}

\section{Proof of Theorem~\ref{thm-main}}
\label{section-proof}

\subsection{A local integral}
\label{subsection-local-integral}
Recall that  $\sigma$ is $\chi^{-1}$-generic. We fix a non-zero $\chi^{-1}$-Whittaker functional $\lambda$ of $\sigma$, satisfying
\begin{equation}
\label{eq-lambda-sigma}
\lambda(\sigma(u) v_\sigma)=\chi^{-1}(u)\lambda(v_\sigma), \quad \text{ for all }u\in Z_{r+1}, v_\sigma\in \sigma.	
\end{equation}
Then we define a linear map $\Lambda:\rho\to\pi_0$ by
\begin{equation*}
\begin{split}
\Lambda: \rho=\pi_0\otimes \sigma &\to \pi_0 \\
v_0 \otimes v_\sigma &\mapsto \lambda(v_\sigma)v_0.	
\end{split}
\end{equation*}

Let $\pi$, $\pi_0$ be irreducible admissible representations of $\GSpin(V)$ and $\GSpin(V_0)$ respectively, and let 
\begin{equation*}
\mu: \pi\otimes \pi_0\to \mathbb{C}	
\end{equation*}
be a Bessel functional, i.e., a map in the space $\Hom_{H}(\pi\otimes \pi_0, \xi)$. 

\begin{remark}
\label{remark-central-character-compatibility}
Let $\omega_{\pi}, \omega_{\pi_0}$ denote the central characters of $\pi$ and $\pi_0$ respectively. Note that $F^\times$ lies in both the centers of $\GSpin(V)$ and $\GSpin(V_0)$ (see Lemma~\ref{lemma-center}), and $\xi|_{F^\times}$ is trivial. If $\omega_{\pi}|_{F^\times} \omega_{\pi_0}|_{F^\times}$ is not the trivial character of $F^\times$, then $\Hom_{H}(\pi\otimes \pi_0, \xi)=0$. Thus we may assume that $\omega_{\pi}|_{F^\times} \omega_{\pi_0}|_{F^\times}$ is the trivial character of $F^\times$.
\end{remark}

We have the following result.
\begin{lemma}
\label{lemma-H-invariance-of-mu}
For every $s\in\mathbb{C}$, $v\in \pi$ and $f\in \pi_s^\prime$, the smooth function on $\GSpin(V)$ defined by	
\begin{equation*}
g\mapsto \mu(\pi(g)v, \Lambda(f(g)))
\end{equation*}
is left invariant under $H$.
\end{lemma}

\begin{proof}
Let $g\in \GSpin(V)$ and $b\in H \subset  P_{r+1}^\prime$. We may write $b=nhm$, where $n\in N_{P_{r+1}^\prime}$, $h\in \GSpin(V_0)$, $m\in Z_{r+1}$. Recall that we have identified $Z_{r+1}$ with its lift $\ell_{r+1}(Z_{r+1})\subset \GSpin(V^\prime)$. Then
\begin{equation*}
\Lambda(f(bg))= \Lambda(\rho(hm)	f(g)) = \chi^{-1}(m)\pi_0(h)(\Lambda(f(g))) =\xi^{-1}(b) \pi_0(h)(\Lambda(f(g)))
\end{equation*}
and hence
\begin{equation*}
\begin{split}
	 \mu(\pi(bg)v, \Lambda(f(bg))) &=\xi^{-1}(b) \mu(\pi(bg)v, \pi_0(h)(\Lambda(f(g)))) \\
	 &= \mu( \pi(b^{-1}) \pi(bg)v, \pi_0(b^{-1})\pi_0(h)(\Lambda(f(g)))) \\
	 &=\mu(\pi(g)v, \Lambda(f(g))) 
\end{split}
\end{equation*}
where the last equality holds because $b$ maps to $h$ under the quotient map $H\to \GSpin(V_0)$, and $\pi_0$ is viewed as a representation of $H$ via inflation. 
\end{proof}

Since $\GSpin(V)$ and $H$ are unimodular, there exists a right $\GSpin(V)$-invariant positive measure on $H\backslash \GSpin(V)$ unique up to scalar, and we write $dg$ for such a measure. 
 The following local zeta integral plays an important role in the proof of Theorem~\ref{thm-main}:
\begin{equation}
\label{eq-I-mu}
\mathcal{Z}_{\mu}(f,v):=\int_{H\backslash \GSpin(V)} 	\mu(\pi(g)v, \Lambda(f(g))) dg, \quad f\in \pi_s^\prime, v\in \pi.
\end{equation}
Whenever the integral converges absolutely, the right $\GSpin(V)$-invariance of $dg$ implies that, for every $g_0\in \GSpin(V)$, we have
\begin{equation}
\label{eq-integral-equivariant-property}
\mathcal{Z}_{\mu}(\pi_s^\prime(g_0)f,\pi(g_0)v)=	\mathcal{Z}_{\mu}(f,v).
\end{equation}
Indeed, this follows from the change of variable $gg_0\mapsto g$.

We postpone the proof of the following proposition to Section~\ref{subsection-proof-prop1}.
 
\begin{proposition}
\label{prop-integral-converge-1}
Let $\mu\in \Hom_{H}(\pi\otimes \pi_0, \xi)$ be nonzero. There exist an element $f_\rho\in \pi_s^\prime$ and a vector $v\in \pi$ such that $\mathcal{Z}_{\mu}(f_\rho, v)$ converges absolutely and yields a nonzero number. 
\end{proposition}

We also have the following result.

\begin{proposition}
\label{prop-integral-converge-2}
There exists a number $c_{\mu}\in \mathbb{R}$ such that for all $s\in \mathbb{C}$ with $\mathrm{Re}(s)>c_{\mu}$, the integral $\mathcal{Z}_{\mu}(f,v)$ converges absolutely for all $f\in \pi_s^\prime$ and all $v\in \pi$. In the archimedean case, in the domain of absolute convergence, $\mathcal{Z}_{\mu}(f,v)$ is continuous in the input data.
\end{proposition}

\begin{proof}
We prove this result by reducing from $\GSpin(V)$ to $\SO(V)$. A similar approach has been used in \cite[Proposition 3.1]{KaplanLauLiu2023} and \cite[Theorem 7.1 (i)]{AsgariCogdellShahidi2024}.

Recall that the kernel of the projection map $\pr:\GSpin(V)\to \SO(V)$ is $\ker(\pr)=F^\times$. Since $F^\times\subset H$, the projection induces a homeomorphism
\begin{equation}
\label{eq-prop-integral-converge-2-homeo}
H\backslash \GSpin(V) \cong \pr(H)\backslash \SO(V). 	
\end{equation}
Let $v\in \pi$ and $f\in \pi_s^\prime$, and denote
\begin{equation*}
\mathcal{G}_{\mu, f, v}(g):=\mu(\pi(g)v, \Lambda(f(g))), \qquad g\in \GSpin(V).	
\end{equation*}
By Lemma~\ref{lemma-H-invariance-of-mu}, $\mathcal{G}_{\mu, f, v}$ is left invariant under $H$. Note that $F^\times$ lies in the centers of both $\GSpin(V)$ and $\GSpin(V_0)$. It follows that $\mathcal{G}_{\mu, f, v}$ is invariant under multiplication by $F^\times$. Indeed, for any $z\in F^\times$, we have
\begin{equation}
\label{eq-prop-integral-converge-2-centralinvariant}
\begin{split}
\mathcal{G}_{\mu, f, v}(gz)&=	\mathcal{G}_{\mu, f, v}(zg) \\
&=\mu( \pi(zg)v, \Lambda(f(zg)))\\
&=\mu(\omega_{\pi}(z) \pi(g)v, \omega_{\pi_0}(z) \Lambda(f(g)))  \quad \text{(by \eqref{eq-induced-repn-f(zx)})}\\
&=\omega_\pi(z)\omega_{\pi_0}(z) \mu(\pi(g)v, \Lambda(f(g))) \\
&=\mu(\pi(g)v, \Lambda(f(g))) \\
&= \mathcal{G}_{\mu, f, v}(g)
\end{split}
\end{equation}
where the second to the last equality follows from the assumption that $\omega_{\pi}|_{F^\times} \cdot \omega_{\pi_0}|_{F^\times}$ is the trivial character of $F^\times$ (see Remark~\ref{remark-central-character-compatibility}). 
Thus, the function $\mathcal{G}_{\mu, f, v}$ descends to a function on $\SO(V)$ as follows. For any $\overline{g}\in \SO(V)$, we set 
\begin{equation*}
\overline{\mathcal{G}}_{\mu, f, v}(\overline{g}):=\mathcal{G}_{\mu, f, v}(g),	
\end{equation*}
where $g\in \GSpin(V)$ is any lift of $\overline{g}$. This is well-defined by \eqref{eq-prop-integral-converge-2-centralinvariant}. Moreover, $\overline{\mathcal{G}}_{\mu, f, v}$ is left invariant under $\pr(H)$. Indeed, let $\overline{h}\in \pr(H)$, and choose any lift $h\in H$. Then we have
\begin{equation*}
\overline{\mathcal{G}}_{\mu, f, v}( \overline{h} \overline{g}) =	\overline{\mathcal{G}}_{\mu, f, v}( \overline{hg} ) = \mathcal{G}_{\mu, f, v}(hg) = \mathcal{G}_{\mu, f, v}(g)=\overline{\mathcal{G}}_{\mu, f, v}( \overline{g}).
\end{equation*}
Hence $\overline{\mathcal{G}}_{\mu, f, v}$ is left invariant under $\pr(H)$.

Using the homeomorphism~\eqref{eq-prop-integral-converge-2-homeo}, we may rewrite the local zeta integral $\mathcal{Z}_{\mu}(f,v)$ as 
\begin{equation*}
\begin{split}
\mathcal{Z}_{\mu}(f,v) &= \int_{H\backslash \GSpin(V)} 	\mathcal{G}_{\mu, f, v}(g) dg\\
&= \int_{\pr(H)\backslash \SO(V)} \overline{\mathcal{G}}_{\mu, f, v}(\overline{g}) d\overline{g},
\end{split}
\end{equation*}
where $d\overline{g}$ denotes the quotient measure on $\pr(H)\backslash \SO(V)$ transported from the quotient measure $dg$ on $H\backslash \GSpin(V)$ through the homeomorphism~\eqref{eq-prop-integral-converge-2-homeo}. The absolute convergence of the integral $\int_{\pr(H)\backslash \SO(V)} \overline{\mathcal{G}}_{\mu, f, v}(\overline{g}) d\overline{g}$  and hence of $\mathcal{Z}_\mu(f,v)$, now follows from the absolute convergence of the analogous local integral for special orthogonal groups in \cite[Lemma 4.1]{JiangZhang2014}. Finally, in the archimedean case, the same reduction together with \cite[Proposition 3.4]{JiangSunZhu2010} shows that $\mathcal{Z}_{\mu}(f,v)$ is continuous on the input data in the domain of absolute convergence.
\end{proof}

\subsection{Proof of Proposition~\ref{prop-integral-converge-1}}
\label{subsection-proof-prop1}
The goal of this section is to prove Proposition~\ref{prop-integral-converge-1}, and hence we assume that $\mu\in \Hom_{H}(\pi\otimes \pi_0, \xi)$ is nonzero. By twisting the representation $\sigma$ by a character if necessary, we may assume $s=0$. 

Denote $\overline{P}_r=\{g\in \GSpin(V): gY_r g^{-1}=Y_r\}$. This is the parabolic subgroup of $\GSpin(V)$ opposite to $P_r$ with respect to the decomposition $V=X_r\oplus V_0^\prime\oplus Y_r$. It has a Levi decomposition $\overline{P}_r= (\GL_r \times \GSpin(V_0^\prime))\ltimes  \overline{N}_{P_r}$, where the $\GL_r$-factor is embedded in $\GSpin(V)$ through the lift $\ell_r$ described in Lemma~\ref{lemma-Levi-lifting}, and $\overline{N}_{P_r}$ is the unipotent radical, which is given by
\begin{equation*}
\overline{N}_{P_r}=\{g\in \GSpin(V): \pr(g)y=y \text{ for all }y\in Y_r, \pr(g) w \equiv w \pmod{Y_r} \text{ for all }w\in V_0^\prime \}.	
\end{equation*}
By \cite[Theorem 14.20]{Borel1991}, the multiplication map $P_r \times \overline{N}_{P_r}\to \GSpin(V)$ given by $(p, \overline{n})\mapsto p\overline{n}$ is an $F$-morphic open immersion, and hence the product 
\begin{equation}
\label{eq-product-Pr-overlineNpr}
P_r \overline{N}_{P_r}= \left(N_{P_r} \rtimes (\GSpin(V_0^\prime)\times\GL_r)\right) \overline{N}_{P_r}
\end{equation}
is the open Bruhat cell of $\GSpin(V)$. In particular, it is an open dense subset of $\GSpin(V)$, and its image in $H\backslash \GSpin(V)$ has full measure in $H\backslash \GSpin(V)$. Therefore the integral defining \(\mathcal Z_\mu(f,v)\) may be computed over $H\backslash P_r \overline{N}_{P_r}$. 

Recall that 
\begin{equation*}
H=N_{P_r}\rtimes (\GSpin(V_0)\times Z_r)\subset P_r=	N_{P_r} \rtimes (\GSpin(V_0^\prime)\times \GL_r).
\end{equation*}
It follows that multiplication induces an identification
\begin{equation*}
H\backslash P_r \overline{N}_{P_r} \cong (Z_r\backslash \GL_r)\times (\GSpin(V_0)\backslash \GSpin(V_0^\prime))\times
\overline N_{P_r},	
\end{equation*}
given by $(Z_r m, \GSpin(V_0)g_0^\prime, \overline{n}) \mapsto H \ell_r(m) g_0^\prime \overline{n}$. 
Thus we have
\begin{equation}
\label{eq-sec4-eq0}
\mathcal{Z}_{\mu}(f,v)= \int_{(Z_r\backslash \GL_r)\times (\GSpin(V_0)\backslash \GSpin(V_0^\prime))\times \overline{N}_{P_r}}	\mu(\pi(m g_0^\prime \overline{n} )v, \Lambda(f(m g_0^\prime \overline{n}))) \delta_{P_r}^{-1}(m)dm dg_0^\prime d\overline{n},
\end{equation}
where $\delta_{P_r}$ is the modular character of $P_r$ defined by $p\mapsto |\det(\Ad_{p})|$. We now briefly explain the Jacobian factor in \eqref{eq-sec4-eq0}. Since $\overline{N}_{P_r}$ is unipotent, the $\overline{N}_{P_r}$ part does not contribute to the Jacobian. Moreover, $\delta_{P_r}$ is trivial on the $\GSpin(V_0^\prime)$ part, since $\GSpin(V_0^\prime)$ acts on $\Lie(N_{P_r})$ through the projection to the orthogonal action which has determinant 1. Therefore, the factor $\delta_{P_r}^{-1}(m)$ in the right-hand side of \eqref{eq-sec4-eq0} accounts for the full Jacobian. 
\textbf{Here and in what follows, we identify $\GL_r$ with its image under $\ell_r$ when no confusion can arise.}

 We have the following non-vanishing lemma on $\GSpin(V_0)\backslash \GSpin(V_0^\prime)$.

\begin{lemma}
\label{lemma-sec4-lemma1}
There exist a vector $v\in \pi$ and a smooth function $f_{\pi_0}:\GSpin(V_0^\prime)\to \pi_0$, compactly supported modulo $\GSpin(V_0)$, such that 
\begin{equation*}
f_{\pi_0}(g g_0^\prime)=\pi_0(g)f_{\pi_0}(g_0^\prime), \quad \text{ for all }g\in \GSpin(V_0), \ g_0^\prime\in \GSpin(V^\prime_0),
\end{equation*}
and 
\begin{equation}
\label{eq-lemma-sec4-lemma1-nonvanishing}
\int_{\GSpin(V_0)\backslash \GSpin(V_0^\prime)}  \mu(\pi(g_0^\prime) v, f_{\pi_0}(g_0^\prime))dg_0^\prime\not=0.	
\end{equation}
\end{lemma}

\begin{proof}
Since $\mu\not=0$, we can find $v\in \pi$ and $v_0\in \pi_0$ such that $\mu(v, v_0)=1$. Consider the function $\phi(g_0^\prime):=\mu(\pi(g_0^\prime)v,v_0)$, $g_0^\prime\in \GSpin(V_0^\prime)$. This is continuous, and $\phi(1)=1$. Hence there exists a compact neighborhood $K^\prime$ of the identity in \(\GSpin(V_0^\prime)\) such that $\mathrm{Re}(\phi(h))>0$ for all $h\in K^\prime$. 
Let $\beta:\GSpin(V_0^\prime)\to  \GSpin(V_0)\backslash \GSpin(V_0^\prime)$ be the quotient map. Shrinking $K^\prime$ if necessary, we may assume that $\beta|_{K^\prime}$ is injective, and that there exists an open neighborhood $\Omega$ of the base point $\beta(1)$ in \(\GSpin(V_0)\backslash \GSpin(V_0^\prime)\) together with a smooth section $\widetilde{\beta}:\Omega\to \GSpin(V_0^\prime)$ such that $\widetilde{\beta}(\Omega)\subset K^\prime$. 

Let $\phi_1\in \mathcal{C}_c^\infty(\Omega)$ be a non-negative non-zero smooth function with compact support contained in $\Omega$. 
Define
\begin{equation*}
f_{\pi_0}(g_0^\prime)=
\begin{cases}
\phi_1(x)\,\pi_0(g)v_0, & \text{if } g_0^\prime=g \widetilde{\beta}(x)
\text{ with } g\in \GSpin(V_0),\ x\in \Omega,\\
0, & \text{if } g_0^\prime\notin \GSpin(V_0) \widetilde{\beta}(\Omega).
\end{cases}	
\end{equation*}
Since $\widetilde{\beta}$ is a section of the quotient map, the decomposition $g_0^\prime=g \widetilde{\beta}(x)$ is unique whenever $g_0^\prime\in \GSpin(V_0)\widetilde{\beta}(\Omega)$, so $f_{\pi_0}$ is well-defined. 
By construction, we have
\begin{equation*}
f_{\pi_0}(g g_0^\prime)=\pi_0(g)f_{\pi_0}(g_0^\prime), \quad \text{ for all }g\in \GSpin(V_0), \ g_0^\prime\in \GSpin(V^\prime_0).
\end{equation*}
Moreover, $f_{\pi_0}$ is smooth. Indeed, on the open subset $\GSpin(V_0)\widetilde{\beta}(\Omega)\subset \GSpin(V_0^\prime)$, the map 
\begin{equation*}
\GSpin(V_0)\times \Omega \longrightarrow \GSpin(V_0)\widetilde{\beta}(\Omega),
\qquad
(g,x)\longmapsto g\widetilde{\beta}(x)	
\end{equation*}
is a diffeomorphism in the archimedean case and a homeomorphism onto an open subset in the non-archimedean case. Since $\phi_1$ is smooth and $g\mapsto \pi_0(g)v_0$ is smooth, the function $f_{\pi_0}$ is smooth on the open subset $\GSpin(V_0) \widetilde{\beta}(\Omega)$. Outside $\GSpin(V_0) \widetilde{\beta}(\Omega)$, the function $f_{\pi_0}$ is identically zero. Since $\mathrm{supp}(\phi_1)\subset \Omega$ is compact, the support of $f_{\pi_0}$ is contained in $\GSpin(V_0)\widetilde{\beta}(\mathrm{supp}(\phi_1))$, which is compact modulo $\GSpin(V_0)$, and $f_{\pi_0}$ vanishes near the boundary of $\GSpin(V_0)\widetilde{\beta}(\Omega)$. Hence the extension by zero is smooth on all of $\GSpin(V_0^\prime)$.

Finally, we have
\begin{equation*}
\int_{\GSpin(V_0)\backslash \GSpin(V_0^\prime)}
\mu(\pi(g_0^\prime) v, f_{\pi_0}(g_0^\prime))dg_0^\prime
=
\int_{\Omega}
\phi_1(x)\,\mu(\pi(\widetilde{\beta}(x))v,v_0)dx.	
\end{equation*}
By taking the real parts, we obtain 
\begin{equation*}
\mathrm{Re}\left( \int_{\GSpin(V_0)\backslash \GSpin(V_0^\prime)}
\mu(\pi(g_0^\prime) v, f_{\pi_0}(g_0^\prime))dg_0^\prime\right) = \int_{\Omega}
\phi_1(x) \mathrm{Re} (\phi( \widetilde{\beta}(x) ))dx >0.
\end{equation*}
Thus we have  \eqref{eq-lemma-sec4-lemma1-nonvanishing}.
This finishes the proof of the lemma.
\end{proof}

Take $v\in\pi$ and $f_{\pi_0}$ as in Lemma~\ref{lemma-sec4-lemma1}. For $m\in \GL_{r}$ and $\overline{n}\in \overline{N}_{P_r}$, put
\begin{equation*}
\Phi(m,\overline{n}):=\int_{\GSpin(V_0)\backslash \GSpin(V_0^\prime)} 	\mu(\pi(m g_0^\prime \overline{n} )v,  f_{\pi_0}( g_0^\prime  ) )dg_0^\prime,
\end{equation*}
which is a smooth function on $\GL_r\times \overline{N}_{P_r}$. By Lemma~\ref{lemma-sec4-lemma1}, $\Phi(1,1)\not=0$, and hence $\Phi\not=0$. Note that for any $z\in Z_r$, $m\in \GL_{r}$ and $\overline{n}\in \overline{N}_{P_r}$, we have
\begin{equation*}
\begin{split}
\Phi(zm, \overline{n}) &=	\int_{\GSpin(V_0)\backslash \GSpin(V_0^\prime)} 	\mu(\pi( z m g_0^\prime \overline{n} )v,  f_{\pi_0}( g_0^\prime ))dg_0^\prime \\
&=	\int_{\GSpin(V_0)\backslash \GSpin(V_0^\prime)} 	\mu(\pi( z m g_0^\prime \overline{n} )v,  \pi_0(z)f_{\pi_0}( g_0^\prime ))dg_0^\prime \\
&=\chi(z)\Phi(m, \overline{n}),
\end{split}
\end{equation*}
where the second equality follows from the fact that the representation $\pi_0$ of $H$ has trivial restriction to $Z_r$, and the third equality follows from the fact that $\chi|_{Z_r}=\xi|_{Z_r}$ (see \eqref{eq-xi-chi} and Remark~\ref{remark-chi-xi-compatibility}).

Let $W_r$ be a smooth function on $\GL_r$ with compact support modulo $Z_r$ such that 
\begin{equation*}
W_r(zm)=\chi^{-1}(z)W_r(m), \quad \text{ for all }z\in Z_r, m\in \GL_r.	
\end{equation*}
Then for each $\overline{n}\in \overline{N}_{P_r}$, the function on $\GL_r$ given by $m\mapsto \Phi(m, \overline{n})W_r(m)\delta_{P_r}^{-1}(m)$ is left-invariant under $Z_r$. We remind the reader that $\GL_r=\GL_F(X_r)$ is viewed as a subgroup of $\GL_{r+1}=\GL_F(X_{r+1}^{\prime})$ by acting trivially on $e_{r+1}$, and that $Z_r$ (respectively, $Z_{r+1}$) is the unipotent radical of the Borel subgroup of $\GL_r$ (respectively, $\GL_{r+1}$) stabilizing the flag \eqref{eq-flag-Xr} (respectively, the flag \eqref{eq-flag-Xr+1}), as in Section~\ref{subsection-Group-GSpinVprime}.

\begin{lemma}
\label{lemma-sec4-lemma2}
For every $W_r$ as above, there exist a vector $v_\sigma\in \sigma$ such that 
\begin{equation*}
	W_r(m)=\lambda(\sigma(m)v_\sigma), \quad m\in \GL_r.
\end{equation*}
\end{lemma}
\begin{proof}
This is due to Bernstein and Zelevinsky 	\cite{BernsteinZelevinsky1976} when $F$ is non-archimedean,  and Jacquet and Shalika \cite{JacquetShalika1981A} when $F$ is archimedean. See also \cite[Lecture 4]{Cogdell2004}. Note that for any $z\in Z_r$ and $m\in \GL_r$, we have $\lambda(\sigma (zm)v_{\sigma})=\chi^{-1}(z)\lambda(\sigma(m)v_{\sigma})$ by \eqref{eq-lambda-sigma}.
\end{proof}

Let $\phi_2$ be a smooth and compactly supported function on $\overline{N}_{P_r}$. Since $\Phi\not=0$, we can pick $W_r$ and $\phi_2$ such that
\begin{equation}
\label{eq-sec4-eq1}
\int_{(Z_r\backslash \GL_r) \times \overline{N}_{P_r}}	\Phi(m, \overline{n}) W_r(m) \phi_2(\overline{n}) \delta_{P_r}^{-1}(m) dm d\overline{n} \not=0.
\end{equation}

We now define a section $f_\rho\in \pi_s'$. Note that 
\begin{equation*}
	P_{r+1}^\prime\cap \GSpin(V) =N_{P_{r}} \rtimes (\GL_r\times \GSpin(V_0)),
\end{equation*}
and the multiplication map
\begin{equation*}
\begin{split}
	\iota: (N_{P_{r+1}^\prime} \rtimes  (\GL_{r+1} \times \GSpin(V_0^\prime))\ltimes \overline{N}_{P_r} )  &\to \GSpin(V^\prime) \\
	(u,m, g_0^\prime, \overline{n} ) &\mapsto umg_0^\prime \overline{n}
\end{split}
\end{equation*}
is an open embedding. Put
\begin{equation*}
f_{\rho}(x):=
\begin{cases}
	\phi_2(\overline{n})f_{\pi_0}(g_0^\prime)\otimes (\sigma(m) v_\sigma)  &\text{ if }x=\iota(u, m, g_0^\prime, \overline{n}),\\
	0 &\text{ if $x$ is not in the image of $\iota$},
\end{cases}
\end{equation*}
where $v_\sigma\in \sigma$ is as in Lemma~\ref{lemma-sec4-lemma2}. By construction, $f_\rho$ is smooth and compactly supported modulo $P_{r+1}^\prime$. Moreover, it satisfies the equivariance property \eqref{eq-induced-repn-equivariance} as we explain below. 
Let $x=\iota(u,m,g_0^\prime, \overline{n})$ be in the image of $\iota$. For any $u_0\in N_{P_{r+1}^\prime}$, we have
\begin{equation*}
f_{\rho}( u_0 x)=	f_{\rho}(\iota( u_0 u, m, g_0^\prime, \overline{n}))= \phi_2(\overline{n})f_{\pi_0}(g_0^\prime)\otimes (\sigma(m) v_\sigma) =f_{\rho}(x).
\end{equation*}
For any $g_0\in \GSpin(V_0)$, we have
\begin{equation*}
\begin{split}
f_{\rho}( g_0  x) &= f_{\rho}( g_0 u m g_0^\prime \overline{n}) 	 \\
&=f_{\rho} ( (g_0u g_0^{-1})g_0 m g_0^\prime \overline{n})\\
&=f_{\rho} ( (g_0u g_0^{-1})  m  g_0 g_0^\prime \overline{n}) \quad (\text{by Lemma~\ref{lemma-Levi-GL-GSpin-commute}}) \\
&=f_{\rho}(\iota( g_0 u g_0^{-1}, m, g_0g_0^\prime, \overline{n}))\\
&=\rho(g_0) f_{\rho}(x).
\end{split}
\end{equation*}
For any $m_0\in \GL_{r+1}$, we have
\begin{equation*}
\begin{split}
f_{\rho}(m_0 x)  & = f_{\rho} ( m_0 u m g_0^\prime \overline{n}) = f_{\rho}( \iota( m_0 u m_0^{-1}, m_0 m, g_0^\prime, \overline{n}))=\rho(m_0)f_{\rho}(x).	
\end{split}
\end{equation*}
On the other hand, if $x$ is not in the image of $\iota$, then for any $u_0\in N_{P_{r+1}^\prime}$, $g_0\in \GSpin(V_0)$, $m_0\in \GL_{r+1}$, we have that $u_0 g_0 m_0 x$ is not in the image of $\iota$ and hence 
\begin{equation*}
f_{\rho}( u_0 g_0 m_0 x)=f_{\rho}(x)=0.	
\end{equation*}
Thus, $f_\rho\in \pi_s^\prime$ (recall that $s$ is assumed to be zero in this section). For any $m\in \GL_r$, $g_0^\prime\in \GSpin(V_0^\prime), \overline{n}\in \overline{N}_{P_r}$, we have
\begin{equation*}
\Lambda( f_{\rho}(m g_0^\prime \overline{n} ))=\Lambda( f_{\rho}( \iota(1, m, g_0^\prime, \overline{n}) )) = \Lambda( \phi_2(\overline{n})f_{\pi_0}(g_0^\prime)\otimes (\sigma(m) v_\sigma)) = \phi_2(\overline{n}) \lambda( \sigma(m) v_\sigma) f_{\pi_0}(g_0^\prime).
\end{equation*}

We are ready to finish the proof of Proposition~\ref{prop-integral-converge-1}. Starting from \eqref{eq-sec4-eq0}, we compute that
\begin{equation*}
\begin{split}
\mathcal{Z}_{\mu}(f_\rho,v) &= \int_{(Z_r\backslash \GL_r)\times (\GSpin(V_0)\backslash \GSpin(V_0^\prime))\times \overline{N}_{P_r}}	\mu(\pi(m g_0^\prime \overline{n} )v, \Lambda(f_\rho(m g_0^\prime \overline{n}))) \delta_{P_r}^{-1}(m)dm dg_0^\prime d\overline{n}	\\
&= \int_{(Z_r\backslash \GL_r)\times (\GSpin(V_0)\backslash \GSpin(V_0^\prime))\times \overline{N}_{P_r}}  \mu(\pi(m g_0^\prime \overline{n} )v,   \phi_2(\overline{n}) \lambda(\sigma(m)v_\sigma) f_{\pi_0}( g_0^\prime ) ) \delta_{P_r}^{-1}(m)dm dg_0^\prime d\overline{n} \\
&= \int_{(Z_r\backslash \GL_r) \times \overline{N}_{P_r}} \Phi(m,\overline{n}) \phi_2(\overline{n}) W_r(m) \delta_{P_r}^{-1}(m)dm   d\overline{n}
\end{split}
\end{equation*}
which converges to a non-zero number by \eqref{eq-sec4-eq1}. This concludes the proof of Proposition ~\ref{prop-integral-converge-1}.

\subsection{Proof of Theorem~\ref{thm-main}}
\label{subsection-proof-of-thm}
We are now ready to prove Theorem~\ref{thm-main}.

\begin{proof}[Proof of Theorem~\ref{thm-main}]
We are given $\pi, \pi_0$ and a generic character $\xi$ of $H$ as in \eqref{eq-xi-chi} (and hence a unitary generic character $\chi$ of $Z_{r+1}$). We take $\sigma$ to be an irreducible $\chi^{-1}$-generic representation of $\GL_{r+1}$.  For each $\mu\in \Hom_{H}(\pi\otimes \pi_0, \xi)$, we may define the integral $\mathcal{Z}_{\mu}$ as in \eqref{eq-I-mu}.

Let $\mathcal{F}$ be a finite-dimensional subspace of $\Hom_{H}(\pi\otimes \pi_0, \xi)$. By Proposition~\ref{prop-integral-converge-2}, there exists a number $c_{\mathcal{F}}\in \mathbb{R}$ such that for all $\mu\in  \mathcal{F}$ and all $s\in \mathbb{C}$ with $\mathrm{Re}(s)>c_{\mathcal{F}}$, the integral $\mathcal{Z}_{\mu}(f,v)$ converges absolutely for all $f\in \pi_s^\prime$ and all $v\in \pi$, and defines a linear functional on $\pi_s^\prime\otimes \pi$.

By Lemma~\ref{lemma-irreducible}, we may choose one $s$ with  $\mathrm{Re}(s)>c_{\mathcal{F}}$ such that $\pi_s^\prime$ is irreducible. By \eqref{eq-integral-equivariant-property} and Proposition~\ref{prop-integral-converge-1}, we obtain a linear embedding
\begin{equation*}
\mathcal{F}\hookrightarrow \Hom_{\GSpin(V)}(\pi_s^\prime\otimes \pi, \mathbb{C}), \quad \mu\mapsto \mathcal{Z}_{\mu}.	
\end{equation*}
The injectivity follows from Proposition~\ref{prop-integral-converge-1}.
The $\GSpin(V)$-equivariance of $\mathcal{Z}_\mu$ follows from \eqref{eq-integral-equivariant-property}. The linearity of the map $\mu\mapsto \mathcal{Z}_\mu$ follows from the definition of the integral $\mathcal{Z}_\mu$ and Proposition~\ref{prop-integral-converge-2}.
The space $\Hom_{\GSpin(V)}(\pi_s^\prime\otimes \pi, \mathbb{C})$ is at most one-dimensional (see \cite{EmoryTakeda2023} when $F$ is non-archimedean and \cite{EmoryKimMaiti} when $F$ is archimedean). Therefore, $\mathcal{F}$ is at most one-dimensional. Since $\mathcal{F}$ was arbitrary, we conclude that $\Hom_{H}(\pi\otimes \pi_0, \xi)$ is at most one-dimensional. 
\end{proof}

\bibliographystyle{alpha}
\bibliography{References}

\end{document}